\documentclass{amsart}
\usepackage{amssymb,amsmath,amsthm,mathrsfs} 
\usepackage[backend=biber, style=alphabetic, doi=false, url=false, isbn=false, giveninits=true, maxbibnames=99]{biblatex}
\usepackage{graphicx} 
\usepackage{verbatim}
\usepackage{tikz-cd} 
\usepackage{tikz}
\usepackage{hyperref}
\usepackage{tabularx}
\usepackage{filecontents}
\usepackage{geometry}
\usepackage{mathtools}
\usepackage{float}
\usepackage{booktabs}
\usepackage{lipsum}

\let\phi\varphi

\newcommand\tab[1][.5cm]{\hspace*{#1}}
\newtheorem{theorem}{Theorem}[section]
\newtheorem{corollary}[theorem]{Corollary}
\newtheorem{lemma}[theorem]{Lemma}

\newtheorem{remark}[theorem]{Remark}

\newtheorem{conjecture}[theorem]{Conjecture}

\newcommand*{\bbb}[1]{{\mathbb{#1}}}
\newcommand{\C}{\bbb{C}}
\newcommand{\N}{\bbb{N}}
\newcommand{\Q}{\bbb{Q}}
\newcommand{\R}{\bbb{R}}
\newcommand{\Z}{\bbb{Z}}

\newcommand{\Gal}{\text{Gal}}

\DeclareMathOperator{\Nrm}{Nrm}

\newcommand{\Mod}[1]{\ (\mathrm{mod}\ #1)}
\newcommand{\p}{\mathfrak{p}}

\newcommand{\vep}{\varepsilon}
\newcommand{\mO}{\mathcal{O}}

\newcommand{\bp}{\bar{\mathfrak{p}}}

\title{On the Non $p$-Rationality and Iwasawa Invariants of Certain Real Quadratic Fields}

\author{Peikai Qi}
\email{qipeikai@msu.edu}
\address{Michigan State University, East Lansing, Michigan, USA}
\date{\today}

\author{Matt Stokes}
\email{mathewsonstokes@gmail.com}
\address{Randolph College, Lynchburg, Virginia, USA}
\thanks{Thanks to Jie Yang and Preston Wake for reading the draft of the paper}
\begin{document}

\maketitle

\begin{abstract}
Let $p$ be an odd prime, and $m,r \in \Z^+$ with $m$ coprime to $p$.  In this paper we investigate the real quadratic fields $K = \Q(\sqrt{m^2p^{2r} + 1})$.  We first show that for $m < C$, where constant $C$ depends on $p$, the fundamental unit $\vep$ of $K$ satisfies the congruence $\vep^{p-1} \equiv 1 \Mod{p^2}$, which implies that $K$ is a non $p$-rational field.  Varying $r$ then gives an infinite family of non $p$-rational fields.  When $m = 1$ and $p$ is a non-Wieferich prime, we use a criterion of Fukuda and Komatsu \cite{MR0816225} to show that if $p$ does not divide the class number of $K$, then the Iwasawa invariants for cyclotomic $\Z_p$-extension of $K$ vanish.  We conjecture that there are infinitely many $r$ such that $p$ does not divide the class number of $K$.
\end{abstract}

\section{Introduction}

\subsection{$p$-Rational Fields}

Let $F$ be a number field, and $p$ a prime number.  Let $S$ be the set of primes of $F$ above $p$ and let $L$ be the maximal pro-$p$ abelian extension of $F$ unramified outside of $S$. Then 
$ \Gal(L/F)\cong \Z_p^\rho \times \mathcal{T}_F$, where $\rho$ is a positive integer and $\mathcal{T}_F$ is the torsion subgroup of $\Gal(L/F)$ (Leopoldt's conjecture predicts that $\rho=r_2+1$, where $2r_2$ is the number of complex embeddings $F \hookrightarrow \C$). We say $F$ is $p$-rational if $\mathcal{T}_F$ vanishes. The concept of $p$-rationality has its origins in ramification theory, \cite{MR1124802} \cite{nguyen1990arithmetique}\cite{AIF_1986__36_2_27_0}.

For fixed $p$, the general expectation is that among all real quadratic fields, the $p$-rational ones greatly outnumber the non $p$-rational ones (see Section 4.2.3 of \cite{gras2019heuristics}). 
 For instance, Byeon \cite[Theorem 1.1]{Byeon_2001} shows that for $p > 3$
 \[
 \#\left\{ \text{$K$ is real quadratic} \, | \, 0 < D_K < X \text{, and $K$ is $p$-rational}  \right\} \gg \frac{\sqrt{X}}{\log(X)}
 \]
 (see also Ono \cite{Ono_1999}).  It is proved that there are infinitely many $p$-rational totally real fields for a given $p$ \cite{MR0961918}, under the hypothesis of the $abc$-conjecture.  On the other hand, it is expected that for a fixed real quadratic field $K$, the primes $p$ such that $K$ is non $p$-rational are rare (or possibly finite \cite{Gras_2016}).  For a fixed real quadratic field $K$, it is generally difficult to find a particular $p$ such that $K$ is non $p$-rational. However, if we fix $p$, it is easy to construct an infinite family of non $p$-rational real quadratic fields \cite{GRAS_2023},\cite{gras2019heuristics}. Our Theorem \ref{main1} generalizes \cite{gras2019heuristics} and might overlap with \cite{GRAS_2023}. One can also construct a family of real quadratic fields by varying $p$. For example, $\Q(\sqrt{p(p+2)})$ is $p$-rational for odd primes $p$ (see \cite{gras2019heuristics}\cite{benmerieme2021corps} for this kind of result). 

Let $|\cdot|_{\infty}: \R \to \R^+$ be the usual absolute value. 
\begin{theorem}
Let $r\geq 2,m$ be positive integers, $p$ be an odd prime, and $\gcd(p,m)=1$. Denote $K=\Q(\sqrt{m^2p^{2r}+1})$, and assume that \[
    |m|_{\infty}\leq \left|\frac{(1+\binom{p^{r-1}}{2})p^{p^{r-1}-r}}{2^{p^{r-1}}} \right|_{\infty}
    \]
 Then  $K$ is a non $p$-rational field.
\end{theorem}
This is Theorem \ref{main1} in Section \ref{prat}. We also prove that, for a fixed $p$, this gives an infinite family of non $p$-rational real quadratic fields as $r \to \infty$.  These fields were studied by Gras \cite[Section 4.3]{gras2019heuristics} under the assumption that $m^2p^{2r}+1$ is squarefree. We do not make this assumption.

\subsection{Iwasawa Invariants and Greenberg's Conjecture}

Let $p$ be a prime, $K$ a number field, and let $\zeta_{p^n}$ be a primitive $p^n$-th root of unity. Then $\Q(\zeta_{p^{n+1}})/\Q$ has a unique degree $p^n$ sub-extension of $\Q$ denoted as $\Q_n$. Put $\Q_\infty=\bigcup_n \Q_n$ and $K_\infty=K\Q_\infty$. Then $K_\infty/K$ is a $\Z_p$-extension of $K$ which we call the cyclotomic $\Z_p$-extension of $K$. We write $K_n$ to be the $n$-th layer of the $\Z_p$-extension of $K_\infty/K$.

Let $A_n$ be the $p$ primary part of the class group of $K_n$. Then Iwasawa proved \cite{IW} that there are constants $\mu, \lambda,\nu \in \Z$ for all sufficiently large $n \in \Z^+$ with
\[
|A_n| = p^{\mu p^n + \lambda n + \nu}.
\]
  The constants $\lambda$, $\mu$, and $\nu$ are called the Iwasawa invariants for $K_{\infty}/K$. We may also write them as $\lambda_p(K_{\infty}/K)$, $\mu_p(K_{\infty}/K)$, and $\nu_p(K_{\infty}/K)$ when $p, K$ and $K_\infty$ are not clear from context.

When $K/\Q$ is abelian, Ferrero and Washington \cite{FW} showed that $\mu = 0$ for the cyclotomic $\Z_p$-extension of $K$, and Iwasawa conjectured that this should hold for any number field.  On the other hand, much less is known about the $\lambda$ invariant for the cyclotomic $\Z_p$-extension of $K$. In fact, the cyclotomic $\mu$-invariant and $\lambda$-invariant are thought to be zero for any totally real number field.  This is known as Greenberg's conjecture \cite{MR0401702}.  

To find evidence for Greenberg's conjecture, one may try to construct an infinite family of real quadratic fields $K$ where $\lambda_p  = 0$, for a fixed prime $p$.  By a well-known theorem of Iwasawa \cite[Prop 13.22]{MR1421575}, if $p$ is inert in $K/\Q$, totally ramified in $K_{\infty}/K$, and doesn't divide the class number $h_K$ of $K$, then $\lambda = 0$.  Therefore, one may try to find infinitely many fields such that $p$ doesn't split and $p \nmid h_K$.  This has been done by Ono \cite{Ono_1999} for $3< p< 5000$, Ozaki and Taya \cite{Ozaki1997OnTI} and Fukuda and Komatsu \cite{FK3} when $p = 2$, as well as Horie and Nakagawa \cite{MR0958035} when $p = 3$.  These authors also studied the case when $p$ splits, for $p = 2$ and $p = 3$ respectively.  Kraft \cite{KR} has found a relation between the $\lambda$-invariants of certain imaginary and real quadratic fields when $p = 3$, which gives an infinite family of real quadratic fields with vanishing $\lambda$-invariant by Corollary 1.3 of Itoh \cite{IT}. 

Recall that a prime $p$ is called Wieferich prime if $p^2\mid (2^{p-1}-1)$.
\begin{theorem}\label{main 3}
  Let $p \geq 3$ be a non-Wieferich prime and $r\geq 2$, and write $K=\Q(\sqrt{p^{2r}+1})$.    Assume that $p$ doesn't divide the class number of $K$. Then the Iwasawa invariants $\mu$ and $\lambda$ 
 are zero for the cyclotomic $\Z_p$-extension of $K$. 
\end{theorem}
This is Theorem \ref{main 3} in Section \ref{iwasawa inv}. We also prove that the fields $\Q(\sqrt{p^{2r}+1})$ are distinct for distinct values of $r$, except for finitely many cases where $\Q(\sqrt{p^{2r}+1})=\Q(\sqrt{2})$. We expect that, for a fixed $p$, there are infinitely many $r$ such that $p$ doesn't divide the class number of  $ \Q(\sqrt{p^{2r}+1})$, and hence an infinite family satisfying Greenberg's conjecture. We don't know how to prove this, but there may be another approach, as we now explain. 

We applied a numerical criterion developed by Fukuda and Komatsu \cite{MR0816225} to conclude $K=\Q(\sqrt{p^{2r}+1})$ satisfying Greenberg's conjecture in the theorem.  The numerical criterion \cite{MR0816225} is in terms of invariants $n_1$, $n_2$. (See the definition in Section \ref{iwasawa inv}). Roughly speaking, we have $n_1=1$ and $n_2=r$ for $K=\Q(\sqrt{p^{2r}+1})$. It is possible to apply different kinds of numerical criterion in terms of the invariants $n_1$, $n_2$ developed by Fukuda, Komatsu, and other authors \cite{FK}\cite{FK2}\cite{Fukuda1995TheI}. In this way, one may replace the condition on the class number with another condition (see part (b) of Theorem \ref{fukuda and komatsu}). It may be easier to prove that $\Q(\sqrt{p^{2r}+1})$ satisfies this new condition for infinitely many $r$, but the authors were unable to do it. 

\subsection{On $p,r$ such that $\Q(\sqrt{p^{2r}+1})=\Q(\sqrt{2})$}
 We prove that, for a fixed $p$,  the fields $\Q(\sqrt{p^{2r}+1})$ are distinct except for finitely many cases where $\Q(\sqrt{p^{2r}+1})=\Q(\sqrt{2})$. We expect that the case $\Q(\sqrt{p^{2r}+1})=\Q(\sqrt{2})$ should never happen and show that such exceptions are related to divisibility properties of the sequence $\{G_n\}$ defined by $(1+\sqrt{2})^n=G_n+F_n\sqrt{2}$. We prove the following result in the appendix, which may be of independent interest. 

Let $\nu_2$ be the $2$-adic valuation.
\begin{theorem}
       Let $l,m\in \N$.

\[       G_{gcd(l,m)}=\begin{cases}
           gcd(G_l,G_{m}) & \nu_2(l)=\nu_2(m)\\
           1& \nu_2(l)\neq\nu_2(m)
       \end{cases}\]

\end{theorem}

The theorem is Theorem \ref{divisibility} in Appendix. It is worth pointing out that in \cite{mcconnell2024newinfinitefamiliesnonprational}, McConnell constructs an infinite family of non $p$-rational real quadratic fields based on the sequence $d_l(D)$. The sequence $d_l(D)$ also has ``strong divisibility" if the indecies have the same $3$-adic valuation. There may be a connection between these two sequences. 

\subsection{Structure of the Paper}
In Section \ref{general property}, we study some general properties of the field $K=\Q(\sqrt{m^2p^{2r}+1})$. In Section \ref{prat}, we study the $p$-rationality of $K$. In Section \ref{iwasawa inv}, we study Greenberg's conjecture for $K$. In the Appendix, we study properties of the sequence $G_n$. 

\section{Construction of a Family of Real Quadratic Fields}\label{general property}
In this section, we first study the basic properties of the real quadratic field $\Q(\sqrt{m^2p^{2r}+1})$.  Here $m$, $r$ are positive integers and $p$ is an odd prime which is prime to $m$. Let $\varepsilon$ be the fundamental unit of $\Q(\sqrt{m^2p^{2r}+1})$. Our focus throughout this paper will be on the following congruence
\[
\varepsilon^{p-1}\equiv1\Mod{p^2}
\]
which may or may not hold for a given arbitrary real quadratic field.  This congruence appears as a key step in the numerical criterion to determine non $p$-rationality, and whether the Iwasawa invariants $\mu=\lambda=0$. In the next section, we will use the results of this section to construct a family of real quadratic fields that are not $p$-rational. It is also a family of real quadratic fields whose Iwasawa invariant $\mu=\lambda=0$ with a further assumption on the class number. 

We first prepare some easy lemmas which we will need later on.
\begin{lemma}\label{prime to p}
Let $p$ be an odd prime. Let $\varepsilon$ be an element in $\Z_p^*$. Suppose that $t=\pm \varepsilon^k$ for some integer $k$ prime to $p$. Then we have 
\[
\varepsilon^{p-1}\equiv 1\Mod{p^2\Z_p} \Longleftrightarrow
t^{p-1}\equiv 1\Mod{p^2\Z_p}
\]
\end{lemma}
\begin{proof}
          
   Assume we know $\varepsilon^{p-1}\equiv 1\Mod{p^2\Z_p}$. Then $t^{p-1}=\varepsilon^{k(p-1)}\equiv 1 \Mod{p^2\Z_p}$.

    Assume we know $t^{p-1}\equiv 1\Mod{p^2\Z_p}$. Since $\varepsilon^{p-1}\equiv 1\Mod{p\Z_p}$, we may assume $\varepsilon^{p-1}\equiv 1+pl\mod{p^2\Z_p}$ for some $l \in \Z_p$. Hence,
    \[
    \varepsilon^{p(p-1)}\equiv (1+pl)^p\equiv 1 \Mod{p^2\Z_p}.
    \]
   Since $\gcd(k,p)=1$, there exists integer $x$ and $y$ such that $xk+yp=1$. 
   \[
   \varepsilon^{p-1}=\varepsilon^{xk(p-1)}\cdot\varepsilon^{yp(p-1)}\equiv t^{x(p-1)}\cdot1^y\equiv 1\Mod{p^2\Z_p}.
   \]
 
\end{proof}

\begin{lemma}
    Let $D$ be a positive square-free integer and $K=\Q(\sqrt{D})$. Let $t$ be a unit in $\mO_K$. Assume that the odd prime $p$ is splits in $K$ as $p\mO_K=\p\bar{\p}$. Then
    \[ t^{p-1}\equiv 1\Mod{\p^2}\Longleftrightarrow
   t^{p-1}\equiv 1\Mod{\bar{\p}^2}\Longleftrightarrow
   t^{p-1}\equiv 1\Mod{p^2\mO_K}
    \]
\end{lemma}
\begin{proof}
    Let $\bar{t}$ be the conjugate of $t$. Then $t^{p-1}\bar{t}^{p-1}=1$, since $t$ is a unit.  Now,
  \begin{equation*}
  \begin{split}
     t^{p-1}\equiv 1\Mod{\p^2}  & \Longleftrightarrow \bar{t}^{p-1}\equiv 1\Mod{\bar{\p}^2}  \\
       & \Longleftrightarrow t^{p-1}\bar{t}^{p-1}\equiv t^{p-1} \Mod{\bar{\p}^2}\\
       & \Longleftrightarrow 1\equiv  t^{p-1} \Mod{\bar{\p}^2}    
  \end{split}
 \end{equation*}
The last congruence comes the fact that $t^{p-1} -1 \in \p^2 \cap \bp^2 = p^2\mO_K$.
\end{proof}
\begin{remark}
   We can embed $K$ into $\Q_p$ by localization at prime $\p$ or $\bar{\p}$. The first congruence tells us that $t^{p-1}\equiv1 \Mod{p^2\Z_p}$ doesn't depend on the choice of embedding.  Later, we will write $t^{p-1}\equiv1 \Mod{p^2}$ for simplicity. One can understand it as $t^{p-1}\equiv1 \Mod{p^2\Z_p} $ by embedding $K$ into $\Q_p$ or one can understand it as $t^{p-1}\equiv1 \Mod{p^2\mO_K} $. It doesn't matter since they are all equivalent.  
\end{remark}
The proof of the following lemma is inspired by the proof of Theorem 3.4 of \cite{bouazzaoui2019fibonaccisequencesrealquadratic}.
\begin{lemma}\label{Fib to Units}
   Let $D$ be a positive square-free integer and $K=\Q(\sqrt{D})$. Assume that $p > 2$ is split in $K$. Let $t=a+b\sqrt{D}$ be a unit in $\mO_K$,where  $a, b\in \Z[\frac{1}{2}]$. Let $\Nrm$ denote the norm map from $K$ to $\Q$. Assume that $\Nrm(t)=-1$ and $p\mid a$. Then 
   \[
   t^{p-1}\equiv 1\Mod{p^2}\Longleftrightarrow a\equiv0 \Mod{p^2}
   \]
\end{lemma}
\begin{proof}
        Let $\bar{t}=a-b\sqrt{D}$ be the Galois conjugates of $t$. Define the generalized Fibonacci sequence $\{F_n\}_n$ as
\[
F_{n+2}=(t+\bar{t})F_{n+1}-\Nrm(t)F_n= 2aF_{n+1}+F_n
\]
and $F_0=0$ and $F_1=1$. The Binet formula \cite{319878a9-c233-3261-b3af-4fbfdf59d6fd}
tells us 
\[
F_{n}=\frac{t^{n}-\bar{t}^{n}}{t-\bar{t}} \text{ for any } n\geq 1
\]

Computing further into the sequence, we have $F_2=2a, F_3=4a^2+1, F_4=8a^3+4a,$ and $ F_5=16a^4+12a^2+1$. By induction, we have 
\[
F_{2n}\equiv 2na\Mod{a^2}  \tab \text{ and } \tab F_{2n-1}\equiv 1\Mod{a^2}\]
and since we assumed that $p\mid a$, we have 
\[
F_{p-1}\equiv (p-1)a\Mod{p^2}.
\]
Hence,
\[
 F_{p-1}\equiv 0\Mod{p^2}\Longleftrightarrow p^2\mid a
\]
On the other hand, the Binet formula tells us 
\[
(t-\bar{t})F_{p-1}=t^{p-1}-\bar{t}^{p-1}=t^{1-p}(t^{p-1}-1)(t^{p-1}+1)
\]
Since $-1=t\bar{t}=a^2+b^2D$, and $p\mid a$, we have $p\nmid b$ and $p\nmid (t-\bar{t})$. Since $p\nmid t^{1-p}(t^{p-1}+1)$, we have 
\[
t^{p-1}\equiv 1\Mod{p^2}\Longleftrightarrow  F_{p-1}\equiv 0\Mod{p^2}.
\]
\end{proof}

\begin{lemma}\label{p split}
    Let $m^2p^{2r}+1=b^2D$, where $D$ is a square-free integer, $r\geq 2,m,b$ are positive integers, $p\geq 3$ is an odd prime and $\gcd(p,m)=1$. Let $K=\Q(\sqrt{D})$. Then $p$ splits in the real quadratic field $K$. 
\end{lemma}
\begin{proof}
   We have $p$ splits in $K=\Q(\sqrt{D})$ if and only if the Legendre symbol $(\frac{D}{p})=1$. Notice that
    \[
    \left(\frac{D}{p}\right)= \left(\frac{b^2D}{p}\right)=\left(\frac{m^2p^{2r}+1}{p}\right)=\left(\frac{1}{p}\right)=1.
    \]
\end{proof}
Let $|\cdot|_{\infty}$ denote the usual absolute value on $\R$. 
\begin{theorem}\label{general setting}
    Let $m^2p^{2r}+1=b^2D$, where $D$ is a square-free integer, $r\geq 2,m,b$ are positive integers, $p\geq 3$ is an odd prime and $\gcd(p,m)=1$. Let $K=\Q(\sqrt{D})$, and let $\varepsilon_D$ be the fundamental unit of $K$.  Assume that \[
    |m|_{\infty}\leq \left|\frac{(1+\binom{p^{r-1}}{2})p^{p^{r-1}-r}}{2^{p^{r-1}}} \right|_{\infty}
    \]
 Then 
   \[
   \varepsilon_D^{p-1}\equiv 1\Mod{p^2}.
   \]
\end{theorem}
\begin{proof}
    Write $t:=mp^r-b\sqrt{D}$.  Then $\Nrm(t)=(mp^r-b\sqrt{D})(mp^r+b\sqrt{D})=-1$. By Lemma \ref{Fib to Units}, we know that $t^{p-1}\equiv 1\Mod{p^2}$.  If $t=mp^r-b\sqrt{D}=\pm\varepsilon_D^k$ for some integer $k$ prime to $p$, then the conclusion holds by lemma \ref{prime to p}.

 Now assume $t=mp^r-b\sqrt{D}=\pm\varepsilon_D^{p^lk}$ for some integer $k$ prime to $p$ and $l\geq 1$. Let $x+y\sqrt{D}:=\varepsilon_D^k$, where $x,y\in\Z[\frac{1}{2}]$. we have
   \begin{equation} \label{expand}
\begin{split}
\pm(mp^r-b\sqrt{D})= &(x+y\sqrt{D})^{p^l} \\
  =& x^{p^l}+\binom{p^l}{1}x^{p^l-1}y\sqrt{D}+ \binom{p^l}{2}x^{p^l-2}(y\sqrt{D})^2+\binom{p^l}{3}x^{p^l-3}(y\sqrt{D})^3\\
    &+\cdots+ \binom{p^l}{p^l-1}x(y\sqrt{D})^{p^l-1}+ (y\sqrt{D})^{p^l}.
\end{split}
\end{equation}
Then 
\begin{equation} \label{a relation}
   \pm mp^r=x^{p^l}+\binom{p^l}{2}x^{p^l-2}(y\sqrt{D})^2+\cdots+\binom{p^l}{p^l-1}x(y\sqrt{D})^{p^l-1}. 
\end{equation}   
We know $\binom{p^l}{k}\equiv 0\Mod{p}$ for $1< k< p^l$. So, \eqref{a relation} implies
\[
0\equiv x^{p^l} \Mod{p}.
\]
Therefore $p\mid x$. 

Comparing the coefficient of $\sqrt{D}$ in equation \eqref{expand}, we have
\[
\pm b=\binom{p^l}{1}x^{p^l-1}y+\binom{p^l}{3}x^{p^l-3}y^3D+\cdots+y^{p^l}D^{\frac{p^l-1}{2}}.
\]
Reducing both sides modulo $p$, we have 
\[
\pm b\equiv y^{p^l}D^{\frac{p^l-1}{2}}\Mod{p}
\]
Since $p\nmid b$ and $p\nmid D$, we get $p\nmid y$.

Now, assume $p^2\nmid x$. Hence $v_p(x)=1$ where $v_p$ is the $p$-adic valuation. By \cite{51475}, one has $v_p(\binom{p^l}{i})=l-v_p(i)$. Hence, for $i\geq 1$,
\[
v_p\left(\binom{p^l}{i}x^{i}(y\sqrt{D})^{p^l-i}\right)=l-v_p(i)+i\geq l+1
\]
with equality holding if and only if $i=1$, since we have assumed $p$ is an odd prime. We have
\[
r=v_p(\pm mp^r)=v_p(\sum_{odd\: i}\binom{p^l}{i}x^{i}(y\sqrt{D})^{p^l-i})=l+1.
\]
Now, taking the real absolute value $|\cdot|_{\infty}$ of both side of equation \ref{a relation}, 
\begin{equation*}
    \begin{split}
        |mp^r|_{\infty}&=|x^{p^l}+\binom{p^l}{2}x^{p^l-2}(y\sqrt{D})^2+\cdots+\binom{p^l}{p^l-1}x(y\sqrt{D})^{p^l-1}|_{\infty}\\
        &> |x^{p^l}+\binom{p^l}{2}x^{p^l-2}(y^2D)|_{\infty}\\
        &=|x^{p^l}+\binom{p^l}{2}x^{p^l-2}(x^2+1)|_{\infty}\\
        &>|(1+\binom{p^l}{2})x^{p^l}|_{\infty}.
    \end{split}
\end{equation*}

We know $x\in \Z[1/2]$, $p\mid x$, and $x\neq0$. Hence $|x|_{\infty}\geq \frac{p}{2}$. Since $r=l+1$ 
\[
|m|_{\infty}> \left|\frac{(1+\binom{p^l}{2})x^{p^l}}{p^r}\right|_{\infty}>\left|\frac{(1+\binom{p^{r-1}}{2})p^{p^{r-1}-r}}{2^{p^{r-1}}}\right|_{\infty}
\]
which contradicts our assumption on $m$. Hence we must have $p^2|x$.

Now, we know $p^2\mid x$. By definition,
\[
-1=(\pm t)(\pm \bar{t})=(x+y\sqrt{D})^{p^l}(x-y\sqrt{D})^{p^l}.
\]
Hence, \[
(x+y\sqrt{D})(x-y\sqrt{D})=-1.
\]
The norm $\Nrm(x+y\sqrt{D})=-1$. By Lemma \ref{Fib to Units}, we know $(x+y\sqrt{D})^{p-1}\equiv 1\Mod{p^2}$. Then by Lemma
\ref{prime to p}, we have $\varepsilon_D^{p-1}\equiv 1\Mod{p^2}$.
\end{proof}
\begin{remark}
    The condition on the upper bound of $|m|_{\infty}$ is only used in the case that $t=mp^r-b\sqrt{D}$ is the $p$th-power of a unit. Empirically speaking, this case is rare. When $m$ is odd, 
    \[b^2D=1+m^2p^{2r}\equiv1+1=2\Mod{4}\]
    implies that $ D\equiv2\mod{4}$, and we can assume $x\in \Z$. This allows us to further relax the upper bound. However, the upper bound is enough for our purposes since we take $m=1$ in a later section. 
\end{remark}
We take $m=1$ in the next lemma. It is a generalization of Exercise 8.3.3 in \cite{PANT}.
\begin{lemma}\label{fundamental units}
        Let $p^{2r}+1=b^2D$, where $D$ is a square-free integer, $r\geq 2, b$ are positive integers, and $p\geq 3$ is an odd prime. Let $K=\Q(\sqrt{D})$. Then $p^r+b\sqrt{D}$ is a fundamental unit for $K$ except for the finitely many choices of $p$ and $r$ such that $D=2$.
\end{lemma}
\begin{proof}
    Note that $D \equiv 2 \Mod 4$, so the ring of integers of $\Q(\sqrt{D})$ is $\Z[\sqrt{D}]$.  Now, we know $t = p^r + b\sqrt{D}$ is a unit.  Let $\varepsilon_D = x + y\sqrt{D}$ be the fundamental unit (without loss of generality we may take $x,y >0$), and suppose $t= \varepsilon_D^k$ for some $k \in \Z^+$. We will prove that $k=1$ except for a finite number of choices for $p$ and $r$.

    Assume $k$ is even.   Then
    \[
    -1=\Nrm(p^r+b\sqrt{D})=\Nrm(\varepsilon_D^k)=(\pm 1)^k=1
    \]
    which is a contradiction. This also implies that $\Nrm(\varepsilon_D)=-1$.

Now, since $k$ is odd, let $q$ be an odd prime such that $q\mid k$. Write $v + w\sqrt{D} = \varepsilon_D^{k/q}$.  Then $t = (v + w\sqrt{D})^q$, and 
\[
p^r = \sum_{i = 0}^{(q-1)/2} {q \choose 2i} w^{2i}D^i v^{q - 2i} = v u
\]
where
\[
u = \sum_{i = 0}^{(q-1)/2} {q \choose 2i} w^{2i}D^i v^{q - 2i -1}.
\]
Note that we can factor out $v$ since $q$ is odd.  Notice  $v^2 - w^2D = \Nrm(\varepsilon_D^{k/q})=-1$ implies that $\gcd(v,w)=\gcd(v,D)=1$.  Hence $\gcd(v,u) = 1$. Now, $p^r= uv$, so it must be that $v = 1$ and $u = p^r$ since $u > 1$.  But then $v^2 - w^2D =-1$ so that $w^2 D = 2$ which implies $D = 2$ and $w = 1$. We have
\[
1+\sqrt{2}=v + w\sqrt{D} = \varepsilon_D^{k/q}
\]
Hence, $\varepsilon_D=1+\sqrt{2}$ and $k=q$. Now we have that 
\[ (1+\sqrt{2})^q=p^r+b\sqrt{2}\]
Define $G_n, F_n\in \Z$ as $(1+\sqrt{2})^n=G_n+F_n\sqrt{2}$. One can check that $G_{n+2}=2G_{n+1}+G_n$. By the following Theorem in \cite{petho}, there is only a finite choice of $n$ such that $G_n$ is a perfect power.
\end{proof}

\begin{theorem}[Theorem 7 in \cite{petho}]\label{LRS}
    Let $G_n$ be a non-degenerated second-order linear recurrence sequence and $G_{n+2}=A_1G_{n+1}+A_2G_n$. Then there exist effectively computable positive constants $c_1$ and $c_2$ depending only on $G_0,G_1,A_1, A_2$ such that if for the integers $n,x,d$ such that $d\geq 2$ the equation 
    \[
    G_n=x^d
    \]
    holds, then:
\begin{enumerate}
    \item If $|x|>1$, then $\max\{|x|,n,d\}<c_1$
    \item If $|x|\leq 1$, then $n<c_2$.
\end{enumerate} 
\end{theorem}
\begin{remark}
    Theorem \ref{LRS} is proved independently by \cite{MR0666345} and \cite{MR0697495}. The explicit bound in Theorem \ref{LRS} is usually very large. 
\end{remark}
\begin{remark}
    In our case, $G_{n+2}=2G_{n+1}+G_n $ and $G_1=1, G_2=3$. Notice that Theorem \ref{LRS} doesn't require $x$ to be prime. We conjecture that $G_n=p^r$ has no solution $(n, p,r)$ with $p>2$ and $r\geq 2$. The authors are unable to prove this conjecture. However, we can prove that for a fixed odd $p$, there is at most one solution (see Lemma \ref{at most one} in the Appendix). It is worth pointing out that, in \cite{mcconnell2024newinfinitefamiliesnonprational}, McConnell constructed an infinite family of non $p$-rational real quadratic fields by using a kind of strong divisibility property of the sequence $d_l(D)$ (see the Appendix \ref{appendix} for the definition $d_l(D)$). In the Appendix, we prove that $\{G_n\}_n$ is also a strong divisibility sequence when indexes have the same $2$-adic valuation.
\end{remark}
\begin{lemma}\label{diff field}
   For $i=1,2$, let $p^{2r_i}+1=b_i^2D_{r_i}$, where $D_{r_i}$ is a square-free integer, $r_i\geq 2, b_i$ are positive integers, and $p\geq 3$ is an odd prime. Except for finitely many choices of $p_i$ and $r_i$, we have $\Q(\sqrt{D_{r_1}})=\Q(\sqrt{D_{r_2}})$ if and only if $r_1=r_2$. 
\end{lemma}
\begin{proof}
    By lemma \ref{fundamental units}, except for a finite number of cases,  the fundamental unit of $\Q(\sqrt{D_{r_1}})$ is $p^{r_1}+b_1\sqrt{D_{r_1}}$ and the fundamental unit of $\Q(\sqrt{D_{r_2}})$ is $ p^{r_2}+b_2\sqrt{D_{r_2}}$. Hence $\Q(\sqrt{D_{r_1}})=\Q(\sqrt{D_{r_2}})$ implies that
    \[
    p^{r_1}+b_1\sqrt{D_{r_1}}=\pm (p^{r_2}+b_2\sqrt{D_{r_2}})^{\pm1}.
    \]
    We know that $(p^{r_2}+b_2\sqrt{D_{r_2}})^{-1}=-p^{r_2}+b_2\sqrt{D_{r_2}}$. Hence we must have $r_1=r_2$.
\end{proof}

By lemma \ref{fundamental units} and lemma \ref{diff field}, for a fixed odd prime $p$,  the family of real quadratic fields $\{\Q(\sqrt{p^{2r}+1})|r\in \N, r\geq 2\}$ is an infinite family. Next, we will consider the $p$-rationality and Iwasawa invariants $\mu,\lambda$ for $\Q(\sqrt{m^2p^{2r}+1})$ in the following sections. 

\section{$p$-rationality}\label{prat}

We first recall the definition of a $p$-rational of a number field.  It was introduced by Movahhedi and Nguyen Quang Do \cite{MR1124802} \cite{nguyen1990arithmetique}\cite{AIF_1986__36_2_27_0}. The concept has been developed and related to many arithmetic subjects (see \cite{gras2003class}). There are many equivalent definitions of $p$-rationality for number fields. See also Section 3.3.2 of \cite{mcconnell2024newinfinitefamiliesnonprational} for some equivalent definitions for $p$-rationality of real quadratic fields. 

Let $K$ be a number field and $p$ be an odd prime. Let $L$ be the maximal abelian pro-$p$ extension of $K$ unramified outside $p$. Suppose $K$ has $r_1$ real embeddings and $2r_2$ complex embeddings. Assume Leopoldt's conjecture holds for $K$. Then there are only $r_2+1$ independent $\Z_p$-extensions of $K$. Hence
\[
\Gal(L/K)\cong \Z_p^{r_2+1}\times \mathcal{T}_K
\]
where $\mathcal{T}_K$ is a finite abelian $p$-group. We say $K$ is $p$-rational if and only if $\mathcal{T}_K=0$. The size of $\mathcal{T}_K$ is related to the special value of the $p$-adic $L$-function corresponding to $K$. The following theorem is due to Coates (see the Appendix of \cite{MR0460282}).
\begin{theorem}[Coates]\label{Coates}

    Let $K$ be a totally real field and $p$ be an odd prime. Assume Leopoldt's conjecture holds for $K$. Then $\#\mathcal{T}_K$ has the same $p$-adic valuation as 
    \[
    \frac{w_1(K(\mu_p))h_KR_p(K)\prod_{\p\mid p}(1-(\Nrm\p)^{-1})}{\sqrt{\Delta_{K/\Q}}}
\]
Here $\mu_p$ is the group of $p$-th roots of unity, $w_1(K(\mu_p))$ is the number of roots of unity of $K(\mu_p)$, $h_K$ is the class number of $K$, $R_p(K)$ is the $p$-adic regulator of $K$, $\Nrm\p$ is the absolute norm of $\p$, and  $\Delta_{K/\Q}$ is the discriminant of $K$.
\end{theorem}

\begin{theorem}\label{main1}
Let $r\geq 2,m$ be positive integers, $p$ be an odd prime, and $\gcd(p,m)=1$. Denote $K=\Q(\sqrt{m^2p^{2r}+1})$, and assume that \[
    |m|_{\infty}\leq \left|\frac{(1+\binom{p^{r-1}}{2})p^{p^{r-1}-r}}{2^{p^{r-1}}} \right|_{\infty}.
    \]
 Then  $K$ is a non $p$-rational field.
\end{theorem}
\begin{proof}
    We will use Theorem \ref{Coates} of Coates. Since $K$ is a real quadratic field, $w_1(K(\mu_p))=p$. By Lemma \ref{p split}, $p$ splits in $K$. Hence, $\prod_{\p\mid p}(1-(\Nrm\p)^{-1})=(p-1)^2/p^2$. Let $\varepsilon_D$ be the fundamental unit of $K$. Notice that the $p$-adic regulator $R_p(K)=\log_p(\varepsilon_D)\equiv 0\Mod{p^2}\Longleftrightarrow \varepsilon_D^{p-1}\equiv 1\Mod{p^2}$. By Lemma \ref{general setting}, we know that $v_p(R_p(K))\geq 2$. Since $\gcd(p,m^2p^{2r}+1)=1$, the discriminant $\Delta_{K/\Q}$ is prime to $p$.  Putting everything together,
    \begin{equation*}
    \begin{split}
         v_p \left( \omega(K(\mu_p))\prod_{\p\mid p}(1-\Nrm(\p)^{-1})\frac{R_p(K)h_K}{\sqrt{\Delta_{K/\Q}}}\right)=&v_p(\omega(K(\mu_p)))+v_p\left(\prod_{\p\mid p}(1-\Nrm(\p)^{-1})\right)\\
         &+v_p(R_p(K))+v_p(h_K)-v_p\left(\sqrt{\Delta_{K/\Q}}\right)\\
         \geq&1-2+2+v_p(h_K)-0\\
         \geq& 1.
    \end{split}    
    \end{equation*}   
Hence, $v_p(\mathcal{T}_K)\geq 1$ by Theorem \ref{Coates}, and therefore the real quadratic field $K$ is a non $p$-rational field.

\end{proof}

\begin{corollary}\label{inf family}
    For a fixed odd prime $p$, there is an infinite family of non $p$-rational real quadratic fields.
\end{corollary}
\begin{proof}
    By Theorem \ref{main1}, the members of the following set are non $p$-rational real quadratic fields.
    \[
   \left \{\Q(\sqrt{m^2p^{2r}+1}) \, \big| \, r\geq 2,m \in \Z^+, \gcd(p,m)=1,  |m|_{\infty}\leq \left|\frac{(1+\binom{p^{r-1}}{2})p^{p^{r-1}-r}}{2^{p^{r-1}}} \right|_{\infty} \right\}
    \]
    This set contains the following subset,
    \[
  \{\Q(\sqrt{p^{2r}+1}) \, | \, r\geq 2\}.
    \]
By Lemmas  \ref{fundamental units} and \ref{diff field}, this subset is an infinite set. Hence, we have an infinite family of non $p$-rational real quadratic fields.
\end{proof}
The family of Corollary \ref{inf family} overlaps with some families found in \cite{GRAS_2023} and \cite{gras2019heuristics}. In Section 4.3 of \cite{gras2019heuristics}, Gras considers the real quadratic fields $\Q(\sqrt{m^2p^{2r}+1})$, but he assumes that $m^2p^{2r}+1$ is square-free.

\section{Iwasawa Invariants}\label{iwasawa inv}
In this section, we will consider the Iwasawa invariants $\mu$ and $\lambda$ of the cyclotomic $\Z_p$-extension of $\Q(\sqrt{p^{2r}+1})$. In \cite{MR0401702}, Greenberg conjectured that the Iwasawa invariant $\mu=\lambda=0$ for the cyclotomic $\Z_p$-extension for any totally real field. To test the conjecture, other authors have developed numerical criterion to determine when $\mu=\lambda=0$ for real quadratic fields (see for example \cite{FK}\cite{FK2}\cite{Fukuda1995TheI}\cite{MR1373702}\cite{MR1854114}).  It is then natural to use these criteria to construct infinite families of real quadratic fields such that Greenberg's conjecture holds (see for example \cite{Ozaki1997OnTI}\cite{FK3}\cite{MR0958035}\cite{KR}). In this section, we first recall a numerical criterion developed by Fukuda and Komatsu \cite{MR0816225}. Then we apply the criterion to the real quadratic field $\Q(\sqrt{p^{2r}+1})$. 

Let $K$ be any real quadratic field and $p$ be an odd prime. Let $h_K$ be the class number of $K$. Assume that $p$ splits in $K$ as $p\mO_K=\p\bar{\p}$. Let $\alpha$ be a generator of the principal ideal $\bar{\p}^{h_K}$. Let $\varepsilon_K$ be the fundamental unit of $K$. Define the integers $n_1$ and $n_2$ such that 
\[
\alpha^{p-1}\equiv 1\Mod{\p^{n_1}} \tab \text{ and } \tab \alpha^{p-1}\not\equiv 1\Mod{\p^{n_1+1}}
\]
\[
\varepsilon_K^{p-1}\equiv 1\Mod{\p^{n_2}} \tab \text{ and } \tab \varepsilon_K^{p-1}\not\equiv 1\Mod{\p^{n_2+1}}
\]
Despite our choice of $\alpha$, the value of $n_1$ is uniquely determined under the condition $n_1\leq n_2$. Fukuda and Komatsu, among other authors, have developed a series of criterion in terms of invariant $n_1, n_2$ to check Greenberg's conjecture for real quadratic fields \cite{FK}\cite{FK2}\cite{Fukuda1995TheI}.  Here we only recall one criterion from \cite{MR0816225}.

Let $K=K_0\subset K_1\subset K_2\cdots\subset K_\infty$ be the cyclotomic $\Z_p$-extension of $K$. Let $A_n$ be the $p$-primary part of the class group of $K_n$. Let $D_n$ be the subgroup of $A_n$ generated by the prime ideals above $p$. 

\begin{theorem}[Fukuda and Komatsu \cite{MR0816225}]\label{fukuda and komatsu}
Let $K$ be a real quadratic field and $p$ an odd prime which splits in $K/\Q$. Assume that 
\begin{enumerate}
    \item $n_1=1$
    \item $A_0=D_0$
\end{enumerate}
Then, for $n\geq n_2-1$, we have $|A_n|=|D_n|=|D_0|\cdot p^{n_2-1}$.
\end{theorem}
    Let $p^{2r}+1=b^2D$, where $D$ is a square-free integer, $r\geq 2, b$ are positive integers, and $p\geq 3$ is an odd prime. Let $K=\Q(\sqrt{D})$. By lemma \ref{p split}, we know $p$ splits in $K$, say $p\mO_K=\p\bar{\p}$. Notice that we can factor $p^{2r}=b^2D-1=(b\sqrt{D}+1)(b\sqrt{D}-1)$ in $\mO_K$. Hence, we can take $\p^{2r}=(b\sqrt{D}-1)$ and $\bar{\p}^{2r}=(b\sqrt{D}+1)$. Let $\varepsilon_D$ be the fundamental unit of $K$. Then by Theorem \ref{general setting}, we know 
   \[
   \varepsilon_D^{p-1}\equiv 1\Mod{p^2}
   \]
   and therefore $n_2\geq 2$. In fact, we can further determine the value of $n_2$ if $D\neq 2$.
   \begin{lemma}\label{n2 value}
       $n_2=r$ except for finitely many choices of $p$ and $r$.
   \end{lemma}
\begin{proof}
    By Lemma \ref{fundamental units}, except for a finite choice of $p$ and $r$ such that $D=2$, we can take $p^r+b\sqrt{D}$ as the fundamental unit of $K$. Notice that $\p^{2r}=(b\sqrt{D}-1)$ tells us $b\sqrt{D}\equiv 1\Mod{\p^{2r}}$. Since $r\geq 2$, we have $2r> r+1>r$.  Therefore,
    \[
    (p^r+b\sqrt{D})^{p-1}\equiv (p^r+1)^{p-1}\equiv 1^{p-1}=1\Mod{\p^{r}}
    \]
    but
    \[
    (p^r+b\sqrt{D})^{p-1}\equiv (p^r+1)^{p-1}\equiv \sum_{i = 0}^{p-1}{p-1 \choose i} p^{ri}\equiv 1+(p-1)p^r \not\equiv 1\Mod{\p^{r+1}}.
    \]  
\end{proof}
\begin{lemma}\label{n1 value}
    Keeping the same setup as before, we have $(b\sqrt{D}+1)^{p-1}\equiv 2^{p-1}\Mod{\p^2}$.  
\end{lemma}
\begin{proof}
    Since we have $b\sqrt{D}\equiv 1\Mod{\p^{2r}}$, 
    \[
    (b\sqrt{D}+1)^{p-1}\equiv(1+1)^{p-1}= 2^{p-1}\Mod{\p^2}.
    \]
\end{proof}
Recall that a prime number $p$ is called a Wieferich prime if $p^2\mid (2^{p-1}-1)$. These primes are thought to be sparse (the only known Wieferich primes are 1093 and 3511 \cite{enwiki:1228062355}).  Nevertheless, it is conjectured that there is an infinite number of Wieferich primes.  We also note that Silverman \cite{MR0961918} has shown the infinitude of non-Wieferich primes under the assumption of the $abc$-conjecture. 
\begin{theorem}\label{main 3}
  Let $p \geq 3$ be a non-Wieferich prime and $r\geq 2$, and write $K=\Q(\sqrt{p^{2r}+1})$.    Assume that $p$ doesn't divide the class number $h_K$. Then the Iwasawa invariants $\mu,\lambda$ 
 are zero for the cyclotomic $\Z_p$-extension of $K$. 
\end{theorem}
\begin{proof}
Let $s$ be the order of the class of $\bar{\p}$ in the class group of $K$ and $\bar{\p}^s = (t)$ for some $t\in K^*$. Since $\bar{\p}^{2r}=(b\sqrt{D}+1)$, then for some integer $k_1$,
\[
b\sqrt{D}+1=\pm t^{2r/s}\varepsilon_K^{k_1}
\]
By Lemma \ref{n2 value}, $\varepsilon_K^{p-1}\equiv 1\Mod{\p^2}$, and by Lemma \ref{n1 value}, $(b\sqrt{D}+1)^{p-1}\equiv 2^{p-1}\not\equiv 1\Mod{\p^2}$. Hence $t^{p-1}\not \equiv 1\Mod{\p^2}$. Let $\alpha$ be a generator of $\bar{\p}^{h_K}$. Then for some integer $k_2$, we have
\[
\alpha =\pm t^{h_K/s}\varepsilon_K^{k_2}
\]
Since $\gcd(p,h_K)=1$, by localizing $K$ at $\p$ we can use Lemma \ref{prime to p} conclude that $\alpha^{p-1}\not \equiv 1\Mod{\p^2}$. Thus, $n_1=1$.

Since $\gcd(p,h_K)=1$, we have $D_0=A_0=0$. Therefore, by Theorem \ref{fukuda and komatsu}, the Iwasawa invariant $\mu=\lambda=0$ for the cyclotomic $\Z_p$-extension of $K$.
\end{proof}
\begin{corollary}
    Fix a non-Wieferich prime $p$.  Then Greenberg's conjecture holds for the members of the following set.
    \[
    \left\{\Q(\sqrt{p^{2r}+1})\,|\,r\geq 2, p \text{ doesn't divide the class number of } \Q(\sqrt{p^{2r}+1}) \right\}
    \]
\end{corollary}
We expect that this is an infinite family by computational data. However, we don't know how to prove it. As mentioned before, there are also other kinds of numerical criteria \cite{FK}\cite{FK2}\cite{Fukuda1995TheI} in terms of $n_1$ and $n_2$ defined above.  Note that we assume $p$ does not divide the class number of $K$ so that $A_0 = D_0$, and hence we may apply Theorem \ref{fukuda and komatsu}.  We may weaken this condition by just assuming $A_0 = D_0$ or applying other numerical criteria.  However, the authors still do not know if there are infinitely many fields amongst $\Q(\sqrt{p^{2r} + 1})$ for varying $r$ that satisfy this different condition. We believe the following:

\begin{conjecture}
    For a fixed prime $p> 2$, there are infinitely many $r \in \Z^+$ such that $p$ does not divide the class number of $\Q(\sqrt{p^{2r} + 1})$.
\end{conjecture}

\section{Appendix}\label{appendix}

Define $G_n,F_n\in \Z$ such that 
\[
(1+\sqrt{2})^n=G_n+F_n\sqrt{2}, \tab n\in \Z.
\]
Comparing $(1+\sqrt{2})^{-n}=(\sqrt{2}-1)^n$ and $(1+\sqrt{2})^n$, we have \[
G_{-n}=(-1)^nG_n, F_{-n}=F_n.
\]
\begin{lemma}
For $l,m \in \Z$, we have the identity
       \begin{equation}\label{E hold}
        G_{l+m}=2G_mG_l-(-1)^mG_{l-m}.
    \end{equation}
\end{lemma}
\begin{proof}
    Looking at $(1+\sqrt{2})^{l+m} = (1 + \sqrt{2})^l (1 + \sqrt{2})^m$, we have
    \[
    G_{l+m}+F_{l+m}\sqrt{2}=(G_l+F_l\sqrt{2})(G_m+F_m\sqrt{2})
    \]
    Thus,
    \begin{equation}\label{eq 1}
       G_{l+m}=G_lF_m+2F_lF_m 
    \end{equation} 
On the other hand, since $(1+\sqrt{2})^{l-m} = (1 + \sqrt{2})^l (1 + \sqrt{2})^{-m}$, we have
\[
G_{l-m}+F_{l-m}\sqrt{2}=\frac{G_l+F_l\sqrt{2}}{G_m+F_m\sqrt{2}}=\frac{(G_l+F_l\sqrt{2})(G_m-F_m\sqrt{2})}{(-1)^m}.
\]
Thus,
\begin{equation}\label{eq 2}
  G_{l-m}=(G_lG_m-2F_lF_m)\cdot(-1)^m.
\end{equation}
Combining equations \eqref{eq 1} and \eqref{eq 2}, we get equation \eqref{E hold}.
\end{proof}

\begin{theorem}\label{divisibility}
       Let $l,m\in \N$. If $l,m \in \N$ have the same $2$-adic valuation, then
\[
G_{gcd(l,m)}=gcd(G_l,G_{m})
\]
If the power of 2 dividing $l$ is different from that dividing $m$. Then
\[
gcd(G_l,G_m)=1
\] 
\end{theorem}
\begin{proof}
By equation \eqref{E hold}, we have 
\[
G_l=G_{l-r+r}=2G_{l-r}G_r-(-1)^rG_{l-2r}
\]
Hence 
\[
gcd(G_l,G_r)=gcd(G_{l-2r},G_r)
\]
    For a pair $(l,r)\in \N^2$, we will define an operation on the pair to create a new pair $(l_1, r_1)$ in the following way: Assume that $l\geq r$, and define the new pair by
    \[l_1=\max\{|l-2r|,|r|\}\]
    \[
    r_1=\min\{|l-2r|,|r|\}
    \]  
The operation on the pair $(l,r)$ has the following property,
\[gcd(G_l,G_r)= gcd(G_{l_1},G_{r_1})\]
\[
gcd(l,r)=gcd(l_1,r_1)
\]
and
\[
\max\{l,r\}\geq \max\{l_1, r_1\}
\]
Notice that $l,r$ have the same $2$-valuation if and only if $l_1,r_1$ have the same $2$-valuation.
By repeating the operation on the pairs,
\[
(l,r),(l_1, r_1),\cdots,(l_{n-1}, r_{n-1}),(l_n, r_n)
\]
we get a sequence of pairs until 
\begin{equation}\label{max equal}
   \max\{l_{n-1}, r_{n-1}\}=\max\{l_n, r_n\} 
\end{equation}
Assume that $l_n=r_{n-1}$, and $r_n=|l_{n-1}-2r_{n-1}|$. By formula \eqref{max equal}, we have $l_{n-1}=l_n=r_{n-1}$. Then $gcd(l,r)=gcd(l_{n-1},r_{n-1})=l_{n-1}=r_{n-1}$, so $gcd(G_l,G_r)= gcd(G_{l_{n-1}},G_{r_{n-1}})= G_{gcd(l,r)}$. Notice that $l_{n-1}=r_{n-1}$ implies that $l_{n-1},r_{n-1}$ have the same $2$-valuation. This case happens when $l,r$ have the same $2$-adic valuation.  

Assume that $l_n=|l_{n-1}-2r_{n-1}|$, and $r_n=r_{n-1}$. By formula \eqref{max equal}, we have $l_{n-1}=|l_{n-1}-2r_{n-1}|$. If  $l_{n-1}=-l_{n-1}+2r_{n-1}$, then $l_{n-1}=r_{n-1}$. It is the same as the previous case. If $l_{n-1}=l_{n-1}-2r_{n-1}$, then $ r_n=r_{n-1}=0$. We have $gcd(G_l,G_r)= gcd(G_{l_n},G_{r_n})=1$ since $G_0=1$. Notice that $ l_{n-1}=2r_{n-1}$ implies that $l_{n-1}, r_{n-1}$ have different $2$-valuation. The case happens only when $l,r$ have different $2$-valuation. 
\end{proof}
The theorem shows that the sequence $\{G_n\}$ is similar to a strong divisibility sequence (it is a strong divisibility sequence when indices have the same $2$-adic valuation). In \cite{mcconnell2024newinfinitefamiliesnonprational}, McConnell constructs an infinite family of non $p$-rational real quadratic fields based on the sequence $d_l(D)$ which was defined by McConnell as follows:  If $\varepsilon_D$ is the fundamental unit of $K$, let 
\[
u = 
\begin{cases}
    \varepsilon_D^2 &\text{ if }\Nrm(\varepsilon_D) = -1
    \\
    \varepsilon_D &\text{ else }
\end{cases}
\]
Then $d_l(D) = u^l + u^{-l} + 1$.
The sequence $d_l(D)$ is similarly a strong divisibility sequence when indexes have the same $3$-adic valuation.  (see Section 3.2 of \cite{mcconnell2024newinfinitefamiliesnonprational}).

Though we expect that there is no odd prime $p$ such that $G_n=p^r$ for $r\geq 2$, the following lemma is as far as the authors can get.

\begin{lemma}\label{at most one}
   Fix an odd prime $p$. Then there is at most one solution $G_n=p^r$ for some $r\geq 2$. 
\end{lemma}
\begin{proof}
    By the argument of the proof for Lemma \ref{fundamental units}, if it has a solution then $n=q$ has to be an odd prime. Assume that there are two solutions $G_{q_1}=p^{r_1}$ and $G_{q_2}=p^{r_2}$ and $q_1,q_2$ are prime number. Then by Theorem \ref{divisibility},
    \[
    1=G_1=G_{gcd(q_1,q_2)}=gcd(p^{r_1},p^{r_2})=p^{\min\{r_1,r_2\}}\neq 1
    \]
    which is a contradiction. 
\end{proof}

\printbibliography

\end{document}